\documentclass[11pt]{amsart}
\usepackage[utf8]{inputenc}
\usepackage{amsfonts,amsmath,amssymb,amsthm,hyperref,setspace}
\usepackage[a4paper,margin=1.9cm,top=2.5cm,bottom=2.5cm,centering,vcentering]{geometry}

\numberwithin{equation}{section}
\theoremstyle{plain}
\newtheorem{theorem}{Theorem}[section]
\newtheorem{corollary}[theorem]{Corollary}

\newtheorem{lemma}[theorem]{Lemma}

\title{Ramanujan-type Congruences for Partition $k$-Tuples with $5$-Cores}
\author{Manjil P. Saikia}
\address{Department of Humanities and Basic Sciences, Indian Institute of Information Technology (IIIT) Manipur, Mantripukhri, Imphal 795002, India}
\email{manjil@saikia.in}

\author[A. Sarma]{Abhishek Sarma}
\address{Department of Mathematical Sciences, Tezpur University, Napaam 784028, Assam, India}
\email{abhiraaj002@gmail.com}

\author[P. Talukdar]{Pranjal Talukdar}
\address{Department of Mathematical Sciences, Tezpur University, Napaam 784028, Assam, India}
\email{msp21105@tezu.ac.in}



\begin{document}

\maketitle

\begin{abstract}
    We prove several Ramanujan-type congruences modulo powers of $5$ for partition $k$-tuples with $5$-cores, for $k=2, 3, 4$. We also prove some new infinite families of congruences modulo powers of primes for $k$-tuples with $p$-cores, where $p$ is a prime.
\end{abstract}

\section{Introduction}

A partition of a positive integer $n$ is a finite non-increasing sequence of positive integers $\lambda=(\lambda_1, \lambda_2, \ldots, \lambda_k)$ such that $\sum\limits_{i=1}^k\lambda_i=n$. The number of partitions of $n$ is denoted by $p(n)$. Euler found the generating function of $p(n)$, given by
\[
\sum_{n\geq 0}p(n)q^n=\frac{1}{(q;q)_\infty},
\]
where
\[
(a;q)_\infty:=\prod_{i\geq 0}(1-aq^i), \quad |q|<1.
\]
Thereafter, the subject got a massive push when Ramanujan found some amazing congruences that the partition function $p(n)$ satisfies; this ushered in an era of study of arithmetic properties of the partition function as well as other restricted types of partitions. Ramanujan proved the following identities
\[
p(5n+4)\equiv 0 \pmod 5, \quad p(7n+5)\equiv 0 \pmod 7, \quad \text{and}\quad p(11n+6) \equiv 0 \pmod{11}.
\]
Over the years various authors have studied different classes of partitions in the hope of proving such Ramanujan-type congruences for newer partition types. In this paper we look at the class of partitions called $k$-tuple partitions with $t$-cores and prove some Ramanujan-type congruences that these partitions satisfy.

The Ferrers-Young diagram of a partition is a pattern of dots with $\lambda_i$ dots in the $i$-th row for the partition $\lambda=(\lambda_1, \lambda_2, \ldots, \lambda_k)$ of $n$. The hook number of a dot is the number of dots directly below and to the right of the dot, including the dot itself. A partition is called a $t$-core partition for $t\geq 2$ if the partition has no hook number divisible by $t$.  We further define a partition $k$-tuple of $n$ to be the $k$-tuple of partitions $(\Lambda_1, \Lambda_2, \ldots, \Lambda_k)$, where $\sum\limits_{i=1}^k \Lambda_i=n$. Let $A_{t,k}(n)$ be the number of partition $k$-tuples of $n$ with $t$-cores and its generating function is given by
\begin{equation}\label{1}
    \sum_{n\geq 0}A_{t,k}(n)q^n=\frac{(q^t;q^t)^{kt}_\infty}{(q;q)^k_\infty}.
\end{equation}
These partitions and the generating function are the objects of study in this paper.

There have been several studies involving the congruence properties of $A_{t,k}(n)$. Dasappa \cite{Dasappa} proved the following result
\begin{equation}\label{eq:d}
    A_{5,2}(5^\alpha n+5^\alpha -2)\equiv 0 \pmod{5^\alpha}, \quad \alpha\geq 1.
\end{equation}
In a similar vein, recently Majid and Fathima \cite{MajidFathima} proved the following result
\begin{equation}\label{eq:f}
       A_{5,3}(5^\alpha n+5^\alpha -3)\equiv 0 \pmod{5^\alpha}, \quad \alpha\geq 1. 
\end{equation}
Both of these results were proved using elementary techniques which involved using dissection formulas and induction. We extend these results in the following theorem.
\begin{theorem}\label{thm-3}
For all $n\geq 0$ and $\alpha\geq 1$, we have
\begin{equation}
A_{5,4}(5^{\alpha+1} n+5^{\alpha+1} -4)\equiv 0 \pmod{5^{\alpha+4}}.
\end{equation}
\end{theorem}

 We further find some new infinite family of congruences for $A_{t,k}(n)$ for some general values of $k$ and $t$, as stated in the following results.
 
\begin{theorem}\label{genthm1}
Let $p\geq 5$ be a prime and let $r\in \mathbb{N}$ with $1\leq r\leq p-1$, be such that $24r+1$
is a quadratic nonresidue modulo $p$. Then, for all $n\geq0$, $i \geq1$ and $N\geq1$, we have
\begin{align*}
  A_{p,p^Ni-1}(pn + r)\equiv0 \pmod {p^{N}}.
\end{align*}
\end{theorem}

\begin{theorem}\label{genthm2}
Let $p\geq 5$ be a prime and let $r\in \mathbb{N}$ with $1\leq r\leq p-1$, be such that $8r +1$
is a quadratic nonresidue modulo $p$. Then, for all $n\geq0$, $i \geq1$ and $N\geq1$, we have
\begin{align*}
    A_{p,p^Ni-3}(pn + r)\equiv0 \pmod {p^{N}}.
\end{align*}
\end{theorem}

\begin{theorem}\label{genthm3}
Let $p\geq 5$ be a prime and let $r\in \mathbb{N}$ with $1\leq r\leq p-1$, be the unique value such that $8r +1\equiv 0 \pmod p$. Then, for all $n\geq 0$ and $i \geq 1$, we have
\begin{align*}
A_{p,pi-3}(pn + r)\equiv0 \pmod p.
\end{align*}
\end{theorem}

We also prove some new individual congruences for $A_{5,t}(n)$ for some specific values of $t$.
\begin{theorem}\label{thm1}
For all $n\geq 0$, the following congruences are true
\begin{align}
A_{5,2}(25n+23)&\equiv 0 \pmod{25},\\
A_{5,2}(125n+123)&\equiv 0 \pmod{125},\\
A_{5,3}(25n+22)&\equiv 0 \pmod{5},\\
A_{5,3}(125n+122)&\equiv 0 \pmod{25},\\
\label{eq:1}A_{5,4}(25n+21)&\equiv 0 \pmod{3125},\\
\label{eq:2}A_{5,4}(125n+121)&\equiv 0 \pmod{15625}.
\end{align}
\end{theorem}

Next, we present a congruence result for $A_{t,k}(n)$ modulo powers of primes, which can be also viewed as an existence result for infinite family of congruences. 
\begin{theorem}\label{extthm1}
Let $p$ be a prime, $k\geq1$, $j\geq0$, $N\geq1$, $M\geq1$, and $r$ be integers
such that $1 \leq r \leq p^M -1$. If for all $n\geq0$,
\[
   A_{p,k}(p^Mn + r)\equiv0 \pmod{p^N},
\]
then for all $n\geq 0$ we have
\[
A_{p,p^{M+N-1}i+k}(p^Mn + r)\equiv0 \pmod{p^N}.
\]
\end{theorem}

The following is an easy corollary.
\begin{corollary}\label{cor3}
For all $i \geq 0$ and $n\geq0$, we have
\begin{align}
A_{5,25i+2}(25n+23)&\equiv 0 \pmod{5},\\
A_{5,125i+2}(25n+23)&\equiv 0 \pmod{25},\\
A_{5,3125i+2}(125n+123)&\equiv 0 \pmod{125},\\
A_{5,25i+3}(25n+22)&\equiv 0 \pmod{5},\\
A_{5,625i+3}(125n+122)&\equiv 0 \pmod{25},\\
A_{5,5i+4}(5n+3,4)&\equiv 0 \pmod{5},\\
\label{inf cong 1}A_{5,125i+4}(25n+21)&\equiv 0 \pmod{25},\\
A_{5,25i+4}(25n+21)&\equiv 0 \pmod{5},\\
A_{5,125i+4}(25n+21)&\equiv 0 \pmod{25},\\
A_{5,625i+4}(25n+21)&\equiv 0 \pmod{125},\\
A_{5,3125i+4}(25n+21)&\equiv 0 \pmod{625},\\
A_{5,15625i+4}(25n+21)&\equiv 0 \pmod{3125},\\
A_{5,390625i+4}(125n+121)&\equiv 0 \pmod{15625}.
\end{align}
\end{corollary}

\begin{proof}
 The proofs of the above congruences follow from Theorem \ref{thm1} and Theorem \ref{extthm1} and are similar in nature. Hence, here we only present the proof of \eqref{inf cong 1}.
 
 The case for $i=0$ is true by \eqref{eq:1}. Using Theorem \ref{extthm1}, and the case for $i=0$, we deduce that
\begin{align*}
   \sum_{n\geq0}A_{5,125i+4}(25n+21)q^n\equiv  \sum_{n\geq0}A_{5,4}(25n+21)q^n\equiv 0 \pmod{25},
\end{align*}
which completes the proof.
\end{proof}

The rest of the paper is organized as follows: in Section \ref{sec2} we state some preliminary results that we require for our proofs, Theorem \ref{thm-3} is then proved in Section \ref{sec3}, Theorems \ref{genthm1}, \ref{genthm2} and \ref{genthm3} are proved in Section \ref{sec4}, Theorem \ref{thm1} is proved in Section \ref{sec5}, Theorem \ref{extthm1} is proved in Section \ref{sec6} and finally we close the paper with some concluding remarks and conjectures in Section \ref{sec7}.

\section{Preliminary Results}\label{sec2}

For the sake of brevity, we use the notation $f_k:=(q^k;q^k)_\infty$ throughout the rest of the paper. We also let
\[
R(q)=\frac{(q;q^5)_\infty (q^4;q^5)_\infty}{(q^2;q^5)_\infty(q^3;q^5)_\infty}.
\]
First, we recall the following $5$-dissections of $\dfrac{1}{f_1}$ and $f_1$. We have \cite[Eq. (7.4.14)]{Spirit}
\begin{multline}\label{eq:bcb}
        \frac{1}{f_1}=\frac{f_{25}^5}{f_5^6}\Bigg(\frac{1}{R^4(q^5)}+\frac{q}{R^3(q^5)}+\frac{2q^2}{R^2(q^5)}+\frac{3q^3}{R(q^5)}+5q^4-3q^5R(q^5)+2q^6R^2(q^5) \\
        -q^7R^3(q^5)+q^8R^4(q^5)\Bigg),
    \end{multline}
    and \cite[Theorem 7.4.4]{Spirit}
\begin{equation}\label{f_1}
    f_1=f_{25}\left(\frac{1}{R(q^5)}-q-q^2R(q^5) \right).
\end{equation}

\begin{lemma}\label{lem}
Let $\sum\limits_{n\geq 0}P_4(n)q^n=\dfrac{1}{f_1^4}$. Then
we have \[\sum_{n\geq 0}P_4(5n+1)q^n = 4\frac{f_5^2}{f_1^6}+550q\frac{f_5^8}{f_1^{12}}+12500q^2\frac{f_5^{14}}{f_1^{18}}+78125q^3\frac{f_5^{20}}{f_1^{24}}.
\]
\end{lemma}

\begin{proof}
Using equation \eqref{eq:bcb}, extracting the terms involving $q^{5n+1}$, dividing by $q$ and then replacing $q^5$ by $q$, we arrive at
\begin{equation}\label{step-1}
    \sum_{n\geq 0}P_4(5n+1)q^n=\frac{f_{5}^{20}}{f_1^{24}}\Bigg(\frac{4}{R^{15}(q)}+\frac{418q}{R^{10}(q)}+\frac{1840q^{2}}{{R^5}(q)}+1015q^3-1840q^{4}R^5(q)+418q^{5}R^{10}(q)-4q^{6}R^{15}(q)\Bigg).
\end{equation}
We use the following formula \cite[Theorem 7.4.4]{Spirit}
\begin{equation}
    \label{rel. of f and R}\frac{1}{R(q)^5}-11q-q^2R(q)^5=\frac{f_1^6}{f_5^6},
\end{equation}
to obtain from equation \eqref{step-1}
\begin{equation}\label{16}
    \sum_{n\geq 0}P_4(5n+1)q^n=4\frac{f_5^2}{f_1^6}+550q\frac{f_5^8}{f_1^{12}}+12500q^2\frac{f_5^{14}}{f_1^{18}}+78125q^3\frac{f_5^{20}}{f_1^{24}}.
\end{equation}
\end{proof}

\begin{lemma}\label{lem1}
If $\sum\limits_{n\geq0} Q_{4}(n)q^n = f_{5}^2f_{1}^{14}$, then we have $\sum\limits_{n\geq0} Q_{4}(5n+4)q^n = -15625q^2f_{5}^{14}f_{1}^{2}$.  
\end{lemma}

\begin{proof}
Using equation \eqref{f_1}, we obtain
\begin{multline*}
\sum_{n\geq 0} Q_{4}(n)q^n = f_{5}^2f_{25}^{14}(1/R^{14}(q) - 14 q/R^{13}(q) + 77 q^2/R^{12}(q) - 182 q^3/R^{11}(q) + 910 q^5/R^9(q) \\- 1365 q^6/R^8(q) - 1430 q^7/R^7(q) + 5005 q^8/R^6(q) - 10010 q^{10}/R^4(q) \\ + 3640 q^{11}/R^3(q)  + 14105 q^{12}/R^2(q) - 6930 q^{13}/R(q) -15625 q^{14}\\ + 6930 q^{15} R(q) + 14105 q^{16} R^2(q) - 3640 q^{17} R^3(q) - 10010 q^{18} R^4(q)\\ + 5005 q^{20} R^6(q) + 1430 q^{21} R^7(q) -1365 q^{22} R^8(q) - 910 q^{23} R^9(q) \\+ 182 q^{25} R^{11}(q) + 77 q^{26} R^{12}(q) + 14 q^{27} R^{13}(q) + q^{28} R^{14}(q)).
\end{multline*}
Extracting the terms involving $q^{5n+4}$ and then dividing by $q^4$ and replacing $q^5$ by $q$, we get
\begin{align*}
 \sum_{n\geq0} Q_{4}(5n+4)q^n = -15625q^2f_{5}^{14}f_{1}^{2}. 
\end{align*}
\end{proof}

\begin{lemma}\label{lem2}
If $\sum\limits_{n\geq0} Q_{5}(n)q^n = q f_5^{8}f_1^{8}$, then we have $\sum\limits_{n\geq0} Q_{5}(5n+4)q^n = -125q f_{5}^{8}f_{1}^{8}$.  
\end{lemma}

\begin{lemma}\label{lem3}
If $\sum\limits_{n\geq0} Q_{6}(n)q^n = q^2 f_5^{14}f_1^{2}$, then we have $\sum\limits_{n\geq0} Q_{6}(5n+4)q^n = -f_{5}^{2}f_{1}^{14}$.  
\end{lemma}
\noindent The proofs of Lemma \ref{lem2} and Lemma \ref{lem3} are exactly similar to the proof of Lemma \ref{lem1}. So, we leave them to the reader.

We also recall that 
\begin{align}
  \label{lem4}  f_1 =& \sum_{m=-\infty}^{\infty}
(-1)^mq^{m(3m-1)/2},
\end{align}
and \cite[p.39, Entry 24(iii)]{BerndtIII}
\begin{align}
  \label{lem5}  f_1^3 =& \sum_{m\geq 0}
(-1)^{m}(2m + 1)q^{m(m+1)/2}.
\end{align}
We know that for all primes $p$ and integers $k \geq 1$, we have
\begin{align}
 \label{lem6}   f_1^{p^k}\equiv f_p^{p^{k-1}}\pmod {p^k}.
\end{align}

\section{Proof of Theorem \ref{thm-3}}\label{sec3}

We prove Theorem \ref{thm-3} using elementary $q$-series techniques, remniscent of the proof of the result of Majid and Fathima \cite{MajidFathima}. But, before that we need the following result.

\begin{theorem}
For all integers $\alpha \geq 0$, we have 
\begin{align}\label{18}
\sum_{n\geq 0}A_{5,4}(5^{\alpha + 1}n + 5^{\alpha + 1} - 4) q^n = A_{\alpha} f_5^{2}f_1^{14}+B_{\alpha}q f_5^{8}f_1^{8}+C_{\alpha}q^2 f_5^{14}f_1^{2}+ D_{\alpha}q^3 \sum_{n\geq0} A_{5,4}(n)q^n, 
\end{align}
where $A_0 = 4$, $B_0= 550$, $C_0 = 12500$, $D_0 = 78125$, and for any integer $n\geq1$, $A_n$, $B_n$, $C_n$, and $D_n$ are defined as
\begin{align}
    A_n &= -C_{n-1} + 4 D_{n-1},\label{A}\\
    B_n &= -125B_{n-1} + 550 D_{n-1},\label{B}\\
    C_n &= -15625A_{n-1} + 12500 D_{n-1},\label{C}\\
    D_n &= D_0^{n+1}.\label{D}
\end{align}
\end{theorem}

\begin{proof}
From equation \eqref{1}, we have
\[
 \sum_{n\geq 0}A_{5,4}(n)q^n=\frac{f_5^{20}}{f_1^4}.
\]
From equation \eqref{16}, we have 
\begin{align}
\sum_{n\geq0} A_{5,4}(5n+1)q^n &= 4 f_5^{2}f_1^{14}+550q f_5^{8}f_1^{8}+12500q^2 f_5^{14}f_1^{2}+ 78125q^3 \frac{f_5^{20}}{f_1^{4}} \nonumber \\
& = 4 f_5^{2}f_1^{14}+550q f_5^{8}f_1^{8}+12500q^2 f_5^{14}f_1^{2}+ 78125q^3 \sum_{n\geq0} A_{5,4}(n)q^n. \label{19}
\end{align}
Equation \eqref{19}, is the case for $\alpha = 0$. Now assume that the result holds for all values up to $\alpha+1$ ($\alpha\geq0$). Replacing $n$ by $5n+4$, and by using Lemmas \ref{lem1}, \ref{lem2}, \ref{lem3}, and equation \eqref{19}, we have 
\begin{align*}
&\sum_{n\geq 0}A_{5,4}(5^{\alpha + 2}n + 5^{\alpha + 2} - 4) q^n\\
&\qquad = A_{\alpha}(-15625q^2 f_{1}^2f_{5}^{14}) + B_{\alpha}(-125q f_{1}^8f_{5}^{8}) + C_{\alpha}(- f_{1}^{14}f_{5}^{2}) + D_{\alpha}(4 f_5^{2}f_1^{14}+550q f_5^{8}f_1^{8}\\ & \qquad \qquad  +12500q^2 f_5^{14}f_1^{2}+ 78125q^3 \sum_{n\geq0} A_{5,4}(n)q^n)\\
&\qquad =  (-C_{\alpha} + 4 D_{\alpha})f_{1}^{14}f_{5}^{2} + (-125B_{\alpha} + 550 D_{\alpha})q f_{1}^{8}f_{5}^{8} \\ & \qquad \qquad + (-15625 A_{\alpha} + 12500 D_{\alpha})q^2 f_{1}^2f_{5}^{14} +   D_{\alpha}78125q^3 \sum_{n\geq0}A_{5,4}(n)q^n\\
& \qquad = A_{\alpha+1} f_{1}^{14}f_{5}^{2} + B_{\alpha+1}q f_{1}^{8}f_{5}^{8} + C_{\alpha+1}q^2 f_{1}^{2}f_{5}^{14} + D_{\alpha + 1} q^3 \sum_{n\geq0}A_{5,4}(n)q^n.
\end{align*}
Hence, the result is true by induction.
\end{proof}

We can finally prove Theorem \ref{thm-3} now.
\begin{proof}[Proof of Theorem \ref{thm-3}]
From equations \eqref{A}, \eqref{B}, \eqref{C} and \eqref{D}, we see that
\begin{align*}
 A_{1} & \equiv 0 \pmod{5^5}, & B_{1} &  \equiv 0 \pmod{5^6}, & C_{1} & \equiv 0 \pmod{5^6}, & D_{1} &  \equiv 0\pmod{5^7},\\
A_{2} & \equiv 0 \pmod{5^6}, & B_{2} &  \equiv 0 \pmod{5^7}, & C_{2} &  \equiv 0 \pmod{5^7}, & D_{2}&   \equiv 0 \pmod{5^8},\\
&~~~\vdots &  &~~~\vdots & &~~~\vdots && ~~~\vdots\\
 A_{\alpha} & \equiv 0 \pmod{5^{\alpha+4}}, & B_{\alpha} & \equiv 0 \pmod{5^{\alpha+5}}, & C_{\alpha} & \equiv 0 \pmod{5^{\alpha+5}}, & D_{\alpha}&  \equiv 0 \pmod{5^{\alpha+6}}.
\end{align*}
Now, it is easy to see that equation \eqref{18} implies Theorem \ref{thm-3}.
\end{proof}

\section{Proofs of Theorems \ref{genthm1}, \ref{genthm2} and \ref{genthm3}}\label{sec4}

\begin{proof}[Proof of Theorem \ref{genthm1}]
From the generating function of  $A_{p,p^Ni-1}(n)$ and equation \eqref{lem6}, we have
\[
 \sum_{n\geq0}A_{p,p^Ni-1}(n)q^n= \frac{f_p^{p(p^Ni-1)}}{f_1^{p^Ni-1}}\equiv\frac{f_p^{p(p^Ni-1)}}{f_p^{p^{N-1}i}}f_1\pmod{p^N}.
\]
With the help of equation \eqref{lem4}, we obtain
\[
     \sum_{n\geq 0}A_{p,p^Ni-1}(n)q^n\equiv\frac{f_p^{p(p^Ni-1)}}{f_p^{p^{N-1}i}}\bigg(\sum_{m=-\infty}^{\infty}
(-1)^mq^{m(3m-1)/2}\bigg)\pmod{p^N}.
\]
For some $m$ and $n$, we are interested in finding out whether $m(3m-1)/2 = pn + r$. This is
equivalent to asking whether $24pn + 24r + 1 = (6m- 1)^2$, which implies $24r + 1 \equiv
(6m-1)^2 \pmod p$. However $24r+1$ is a quadratic nonresidue modulo $p$. It follows that
\[
A_{p,p^Ni-1}(pn + r)\equiv0 \pmod {p^{N}}.
\]
\end{proof}

\begin{proof}[Proof of Theorem \ref{genthm2}]
Like before, we have
\[
    \sum_{n\geq0}A_{p,p^Ni-3}(n)q^n= \frac{f_p^{p(p^Ni-3)}}{f_1^{p^Ni-3}}\equiv\frac{f_p^{p(p^Ni-3)}}{f_p^{p^{N-1}i}}f_1^3\pmod{p^N}.
\]
With the help of equation \eqref{lem5}, we obtain
\begin{align}\label{pfgenthm2}
     \sum_{n\geq0}A_{p,p^Ni-3}(n)q^n\equiv\frac{f_p^{p(p^Ni-3)}}{f_p^{p^{N-1}i}}\bigg(  \sum_{m\geq0}
(-1)^{m}(2m + 1)q^{m(m+1)/2}\bigg)\pmod{p^N}.
\end{align}
For some $m$ and $n$, we are interested in finding out whether $m(m + 1)/2 = pn + r$. This is
equivalent to asking whether $8pn + 8r + 1 = (2m+ 1)^2$, which implies $8r + 1 \equiv
(2m+1)^2 \pmod p$. However $8r+1$ is a quadratic nonresidue modulo $p$. It follows that
\[
A_{p,p^Ni-}(pn + r)\equiv0 \pmod {p^{N}}.
\]
\end{proof}

\begin{proof}[Proof of Theorem \ref{genthm3}]
Due to equations \eqref{lem5} and \eqref{pfgenthm2}, we must determine whether $pn + r =
m(m + 1)/2$ for some integers $m$ and $n$. Completing the square and considering
the result modulo $p$ gives $(2m+1)^2 \equiv 8r+1 \equiv 0 \pmod p$. Therefore, $p$ divides $(2m+1)^2$,
implying  that $p$ divides $2m + 1$. Since the coefficient of $q^{m(m+1)/2}$ in the
series representation in equation \eqref{lem5} is exactly $2m + 1$, it follows that the coefficient
we are interested in is congruent to $0$ modulo $p$.
\end{proof}

\section{Proof of Theorem \ref{thm1}}\label{sec5}

The proofs of the congruences are similar in nature. So, we only present proofs of \eqref{eq:1} and \eqref{eq:2}. For others, we just give the corresponding  generating functions.

First, we prove \eqref{eq:1}.
We have, 
\begin{align*}
 \sum_{n\geq0} A_{5,4}(n)q^n= \frac{f_5^{20}}{f_1^{4}}= f_5^{20}\sum_{n\geq0}P_4(n)q^n. 
\end{align*}
Then, extracting the terms involving $q^{5n+1}$ and dividing by $q$ and then replacing $q^5$ by $q$, we obtain
\begin{align*}
     \sum_{n\geq0}A_{5,4}(5n+1)q^n= f_1^{20}\sum_{n\geq0}P_4(5n+1)q^n.
\end{align*}
With the help of Lemma \ref{lem}, we get
\begin{align*}
   \sum_{n\geq0} A_{5,4}(5n+1)q^n
    &= f_1^{20}\bigg(4\frac{f_5^2}{f_1^6}+550q\frac{f_5^8}{f_1^{12}}+12500q^2\frac{f_5^{14}}{f_1^{18}}+78125q^3\frac{f_5^{20}}{f_1^{24}}\bigg)\\
     &= 4f_5^2f_1^{14}+550q f_5^8 f_1^{8}+12500q^2 f_5^{14}f^{2}+78125q^3\frac{f_5^{20}}{f_1^{4}}.
\end{align*}
Using equations \eqref{eq:bcb} and \eqref{f_1}, extracting the terms involving $q^{5n+4}$, dividing by $q$ and then replacing $q^5$ by $q$, we arrive at

\begin{multline*}
     \sum_{n\geq0}A_{5,4}(25n+21)q^n= 3125\bigg( -4 f_1^{14} f_5^2 +100\frac{ f_5^{20}}{f_1^4 R^{15}(q)} +q \left(10450\frac{ f_5^{20}}{f_1^4 R^{10}(q)}-22 f_1^8 f_5^8\right) \\+q^2 \left(46000\frac{ f_5^{20}}{f_1^4 R^5(q)}-62500 f_1^2 f_5^{14}\right) +25375q^3\frac{ f_5^{20} }{f_1^4}\\-46000q^4\frac{ f_5^{20}  R^5(q)}{f_1^4}+10450q^5\frac{ f_5^{20}  R^{10}(q)}{f_1^4}-100q^6\frac{ f_5^{20}  R^{15}(q)}{f_1^4}\bigg),   
\end{multline*}
which on usage of \eqref{rel. of f and R} reduces to 
\begin{align*}
     \sum_{n\geq0}A_{5,4}(25n+21)q^n=3125\bigg(96 f_1^{14} f_5^2+13728 q f_1^8 f_5^8+312480 q^2f_1^2 f_5^{14} +1953125 q^3\frac{ f_5^{20}}{f_1^4}\bigg),
     \end{align*}
which implies \eqref{eq:1}.

Proceeding in a similar way, we can also deduce
\begin{align*}
     \sum_{n\geq0}A_{5,4}(125n+121)q^n=&~15625\bigg(1500004 f_1^{14} f_5^2+214500550 q f_1^8 f_5^8 +4882512500 q^2 f_1^2 f_5^{14}\\
     &+30517578125 q^3\frac{ f_5^{20} }{f_1^4}\bigg)
     \end{align*}
     which proves \eqref{eq:2}.

Now, we give the generating functions which will complete the proofs of the other congruences stated in the result:
\begin{align*}
\sum_{n\geq0}A_{5,2}(25n+23)q^n&=25\bigg(48f_1^4f_5^4+625q\dfrac{f_5^{10}}{f_1^2}\bigg),\\ 
\sum_{n\geq0}A_{5,2}(125n+123)q^n&=125\bigg(1202 f_1^4 f_5^4+15625 q\frac{ f_5^{10} }{f_1^2}\bigg),\\
\sum_{n\geq0}A_{5,3}(25n+22)q^n&=5\bigg(5838 f_1^9 f_5^3+233250q f_1^3 f_5^9 +1953125 q^2\frac{ f_5^{15} }{f_1^3}\bigg),\\
\sum_{n\geq0}A_{5,3}(125n+122)q^n&=25\bigg(3643791 f_1^9 f_5^3+145754625 q f_1^3 f_5^9 +1220703125 q^2\frac{ f_5^{15}}{f_1^3}\bigg).
\end{align*}
This completes the proof of Theorem \ref{thm1}.

\section{Proof of Theorem \ref{extthm1}}\label{sec6}

Without loss of generality, we may assume that $\displaystyle{r=\sum_{j=0}^{M-1}p^jr_j}$ for $0\le r_j\le p-1$, as $\displaystyle{\sum_{j=0}^{M-1}p^jr_j}$ can take any value between 1 and $p^M-1$. For integers  $M\ge1$ (sufficiently large) and $N\ge1$, we have
\begin{align*}
\sum_{n\geq0} A_{p,p^{M+N-1}i+k}(n)q^n = \dfrac{f_p^{p(p^{M+N-1}i+k)}}{f_1^{p^{M+N-1}i+k}}&= \dfrac{f_{p}^{p^{M+N}i}}{f_{1}^{p^{M+N-1}i}}\sum_{n=0} A_{p,k}(n)q^n\\
&\equiv f_{p}^{p^{M+N-2}(p^2-1)i}\sum_{n=0} A_{p,k}(n)q^n \pmod{p^N}.
\end{align*}

Extracting the terms that involve $q^{pn+r_0}$ from the above identity, we obtain 
\begin{align*}
\sum_{n\geq0} A_{p,p^{M+N-1}i+k}(pn+r_0)q^n 
&\equiv f_{1}^{p^{M+N-2}(p^2-1)i}\sum_{n=0} A_{p,k}(pn+r_0)q^n \\
&\equiv f_{p}^{p^{M+N-3}(p^2-1)i}\sum_{n=0} A_{p,k}(pn+r_0)q^n \pmod{p^N}.
\end{align*}
Now, extracting the terms that involve $q^{pn+r_1}$ from the above identity, we obtain
 \begin{align*}
     \sum_{n\geq0} A_{p,p^{M+N-1}i+k}(p^{2}n+r_0+ pr_1)q^n
&\equiv f_{1}^{p^{M+N-3}(p^2-1)i}\sum_{n\geq0}A_{p,k}(p^{2}n+r_0+ pr_1)q^n \\
&\equiv f_{p}^{p^{M+N-4}(p^2-1)i}\sum_{n\geq0}A_{p,k}(p^{2}n+r_0+ pr_1)q^n \pmod{p^N}.
 \end{align*}

From the above identity, we extract the terms that contain $q^{pn+r_2}$, and from the resulting identity, we again  extract the terms that contain $q^{pn+r_3}$ and so on. It can be seen that after the $M$-th extraction using this iterative scheme, we arrive at
\begin{multline}\label{cor1}
    \sum_{n\geq0} A_{p,p^{M+N-1}i+k}(p^{M}n+r_0+ pr_1+\cdots+p^{M-1}r_{M-1})q^n\\
\equiv f_{1}^{p^{N-1}(p^2-1)i}\sum_{n\geq0}A_{p,k}(p^{M}n+r_0+ pr_1+\cdots+p^{M-1}r_{M-1})q^n \pmod{p^N}.
\end{multline}
Therefore, if we assume that $A_{p,k}(p^{M}n+r_0+ pr_1+\cdots+p^{M-1}r_{M-1})=A_{p,k}(p^{M}n+r)\equiv 0 \pmod{p^N}$, then from the above identity, we have
\[
A_{p,p^{M+N-1}i+k}(p^{M}n+r)\equiv 0 \pmod{p^N}.
\]
This completes the proof of Theorem \ref{extthm1}.

We have the following easy corollary, which follows from equation \eqref{cor1} when $M=1$.
\begin{corollary}\label{cor2}
Let $p$ be a prime, $k\geq1$, $j\geq0$, $N\geq1$, and $r$ be integers
such that $1 \leq r \leq p -1$. Then for all $n\geq0$, we have
\[
   \sum_{n\geq0} A_{p,p^{N}i+k}(pn + r)q^n\equiv f_1^{p^{N-1}(p^2 -1)i} \sum_{n\geq0}A_{p,k}(pn + r)q^n \pmod {p^N}.
\]
\end{corollary}

\section{Concluding Remarks}\label{sec7}
\begin{enumerate}
    %\item Experiments suggest that the following congruence is true for $n\geq 0$,
%\begin{equation}\label{eq:5}
      % A_{5,7}(125n+123)\equiv 0 \pmod{25}.
%\end{equation}
%The running time of \texttt{RaduRK} in these cases is longer than the previous cases. However using the procedure call \texttt{RKMan}, one can hope to prove these in a viable time. For more details about this prodecure call we refer the reader to Smoot \cite{Smoot}.

\item We have found several congruences modulo powers of $5$, individual as well as infinite families similar to those stated in Theorem \ref{extthm1}, for $A_{5,k}(n)$ for higher values of $k$. The proofs of these are routine execises similar to the proofs of \eqref{eq:1} and \eqref{eq:2} hence, they are not proved here. For instance, the following congruences are true:
\begin{align}
A_{5,6}(25n+14,19,24)&\equiv 0 \pmod{25},\\
A_{5,6}(125n+119)&\equiv 0 \pmod{125},\\
A_{5,7}(25n+13,18,23)&\equiv 0 \pmod{25},\\
A_{5,7}(125n+118)&\equiv 0 \pmod{125},\\
A_{5,25i+6}(25n+14,19,24)&\equiv 0 \pmod{5},\\
A_{5,125i+6}(25n+14,19,24)&\equiv 0 \pmod{25},\\
A_{5,3125i+6}(125n+119)&\equiv 0 \pmod{125},\\
A_{5,25i+7}(25n+13,18,23)&\equiv 0 \pmod{5},\\
A_{5,125i+7}(25n+13,18,23)&\equiv 0 \pmod{25},\\
A_{5,3125i+7}(125n+118)&\equiv 0 \pmod{125}.
%A_{5,8}(25n+17)&\equiv 0 \pmod{5},\\
%A_{5,8}(25n+12,22)&\equiv 0 \pmod{25},\\
%A_{5,8}(125n+117)&\equiv 0 \pmod{25},\\
%A_{5,9}(25n+11,21)&\equiv 0 \pmod{5},\\
%A_{5,9}(25n+16)&\equiv 0 \pmod{25},\\
%A_{5,9}(125n+116)&\equiv 0 \pmod{125}.\\
 %A_{5,25i+8}(25n+12,17,22)&\equiv 0 \pmod{5},\\
%A_{5,125i+8}(25n+12,22)&\equiv 0 \pmod{25},\\
%A_{5,25i+9}(25n+11,16,21)&\equiv 0 \pmod{5},\\
%A_{5,125i+8}(25n+16)&\equiv 0 \pmod{25}.
\end{align}

\item Looking at the sequence of results in equations \eqref{eq:d}, \eqref{eq:f} and \eqref{eq:1}, it suggests the possibility of there being other such structures modulo higher powers of $5$, and more generally for $A_{t,k}(n)$ modulo higher powers of $t$. Experiments also suggest some sort of cyclic behaviour, which we do not conjecture here, but it would be interesting to know what can be said about $ A_{5,5i+3}(5^\alpha n+5^\alpha -3)\pmod {5^\alpha}$ and $ A_{5,5i+2}(5^\alpha n+5^\alpha -2)\pmod {5^\alpha}$?
\item Experiments suggest some additional infinite family of congruences modulo powers of $5$ which are stronger results than those indicated by Theorem \ref{extthm1}. We present them here as conjectures:
\begin{align}
A_{5,5i+1}(25n+24)&\equiv 0 \pmod{25},\\
A_{5,25i+2}(25n+23)&\equiv 0 \pmod{25},\\
A_{5,125i+2}(125n+123)&\equiv 0 \pmod{125},\\
A_{5,125i+3}(125n+122)&\equiv 0 \pmod{25},\\
A_{5,625i+4}(125n+121)&\equiv 0 \pmod{3125},\\
A_{5,3125i+4}(125n+121)&\equiv 0 \pmod{15625},\\
A_{5,25i+6}(25n+14,19)&\equiv 0 \pmod{25},\\
A_{5,25i+6}(125n+119)&\equiv 0 \pmod{125},
\\A_{5,25i+7}(25n+13,18,23)&\equiv 0 \pmod{25},
\\A_{5,125i+7}(125n+118)&\equiv 0 \pmod{125}.
%A_{5,125i+8}(125n+117)&\equiv 0 \pmod{25},
%A_{5,125i+9}(125n+116)&\equiv 0 \pmod{125}.
\end{align}
\item Many other congruences modulo higher powers of 5 may exist for higher subsequences. It may be interesting to study those congruences. 
\end{enumerate}

\section*{Acknowledgements}

The second author is partially supported by the institutional fellowship for doctoral research from Tezpur University, Napaam, India. The third author was partially supported by the Council of Scientific \& Industrial Research (CSIR), Government of India under the CSIR-JRF scheme. The author thanks the funding agency.

\bibliographystyle{alpha}

\end{document}